\documentclass[10pt, twoside]{amsart}
\usepackage[latin1]{inputenc}
\usepackage[T1]{fontenc}
\usepackage{amsfonts}
\usepackage{amssymb}
\usepackage{amsthm}
\usepackage{latexsym}
\usepackage{mathrsfs}
\usepackage{amsmath}
\usepackage{verbatim}
\usepackage{enumerate}

\numberwithin{equation}{section}
\makeatletter
\newcommand\ackname{Acknowledgements}
\if@titlepage
  \newenvironment{acknowledgements}{%
      \titlepage
      \null\vfil
      \@beginparpenalty\@lowpenalty
      \begin{center}%
        \bfseries \ackname
        \@endparpenalty\@M
      \end{center}}%
     {\par\vfil\null\endtitlepage}
\else
  
\fi
\makeatother

\newcommand{\bfi}{\begin{fig}}
\newcommand{\efi}{\end{fig}}
\newcommand{\btab}{\begin{tab}}
\newcommand{\etab}{\end{tab}}
\newcommand{\barr}{\begin{array}}
\newcommand{\earr}{\end{array}}
\newcommand{\beqq}{\begin{equation}}
\newcommand{\eeqq}{\end{equation}}
\newcommand{\beao}{\begin{align*}}
\newcommand{\eeao}{\end{align*}\noindent}
\newcommand{\beam}{\begin{eqnarray}}
\newcommand{\eeam}{\end{eqnarray}\noindent}
\newcommand{\bdis}{\begin{displaymath}}
\newcommand{\edis}{\end{displaymath}\noindent}

\newcommand{\bbn}{\mathbb{N}}

\newcommand{\bbr}{\mathbb{R}}

\newcommand{\bbs}{\mathbb{S}}
\newcommand{\bbh}{\mathbb{H}}

\newcommand{\eps}{{\epsilon}}
\newcommand{\epst}{{\tilde{\epsilon}}}

\DeclareMathOperator{\Tr}{Tr}
\DeclareMathOperator{\diag}{diag}

\DeclareMathOperator{\dist}{dist}
\DeclareMathOperator{\diam}{diam}
\DeclareMathOperator{\maxd}{maxd}
\newtheorem{Satz}{Satz}[section]
\newtheorem{Theorem}[Satz]{Theorem}
\newtheorem{Korollar}[Satz]{Corollary}

\newtheorem{Proposition}[Satz]{Proposition}

\newtheorem{Lemma}[Satz]{Lemma}

\theoremstyle{definition}
\newtheorem{Definition}[Satz]{Definition}
\newtheorem{Annahme}[Satz]{Assumption}

\theoremstyle{remark}

\begin{document}

\pagenumbering{arabic}
\setcounter{page}{1}

\title{Mixed volume preserving curvature flows in hyperbolic space}
\author{Matthias Makowski}
\date{August 9, 2012}
\address{Matthias Makowski, Universit\"at Konstanz, 78467 Konstanz, Germany}
\def\fuaddress{@uni-konstanz.de}
\email{Matthias.Makowski\fuaddress}
\address{http://www.math.uni-konstanz.de/\~{}makowski/}
\begin{abstract}
We consider curvature flows with a curvature function $F$, which is monotone, symmetric, homogeneous of degree 1 and either convex or concave and inverse concave, and a volume preserving term. For initial hypersurfaces, which are compact and strictly convex by horospheres, we prove long time existence and exponential convergence to a geodesic sphere of the same mixed volume as the initial hypersurface.
\end{abstract}
\maketitle
\tableofcontents

\nopagebreak

\section{Introduction}
\label{Introduction}
Let $n\in \bbn$, $n\geq 2$. We fix $b\in \bbr_-^*$ and set $a:= \sqrt{|b|}$. Let $N_b$ be a \mbox{$(n+1)$-dimensional}, connected, simply connected Riemannian manifold of constant sectional curvature $b$, i.e. $N_b$ is isometric to $\bbh^{n+1}_{\frac{1}{a}}$, the hyperbolic space of radius $\frac{1}{a}$, 
\begin{equation}
	\bbh^{n+1}_{\frac{1}{a}} := \{p \in L^{n+2}: \langle p, p\rangle_L = -\frac{1}{a^2}, p^0 > 0\}.
\end{equation} 
Here $(L^{n+2}, \langle .,.\rangle)$ denotes the $(n+2)$-dimensional Lorentz-Minkowski space. We want to consider a curvature flow in $N_b$, which is then equivalent to consider a curvature flow in $\bbh^{n+1}_{\frac{1}{a}}$.

We show the long time existence and the exponential convergence to a geodesic sphere of the following curvature flow in $\bbh^{n+1}_{\frac{1}{a}}$:
\begin{equation}
\label{floweq}
\begin{split}
	\dot{x} &= (f - F)\, \nu,\\
	x(0) &= x_0,
\end{split}
\end{equation}
where $x_0: \bbs^n \rightarrow \bbh^{n+1}_{\frac{1}{a}}$ is the immersion of an initial, compact, connected, smooth hypersurface $M_0 := x_0(\bbs^n)$ which is furthermore required to be strictly convex by horospheres (this property will be explained further below). $\nu$ is the corresponding outer normal, $F$ is a smooth curvature function evaluated at the principal curvatures of the flow hypersurfaces $M_t$, $x(t)$ denotes the embedding of $M_t$ and $f$ is a volume preserving global term, $f = f_k$, see the definition below. 
We need to provide the definition of convexity by horospheres. However, we only give a rather analytic definition, for more geometric interpretations of this property we refer the reader to the papers \cite{BorGallRelHad}, \cite{BorMiqTot}, \cite{BorMiqCompHad} and \cite{BorVlaAsyHad}.
\begin{Definition}
A hypersurface $M$ in $\bbh^{n+1}_{\frac{1}{a}}$ is called (strictly) convex by horospheres, (strictly) h-convex for short, if its principal curvatures are (strictly) bounded from below by $a$ at each point.
\end{Definition}

Depending on which type of mixed volume has to be preserved, we define the global term similar as in \cite{McCoyMixedAreaGen}, however, we have to modify it for $k>1$ due to the curvature of $\bbh^{n+1}_{\frac{1}{a}}$:
\begin{equation}
\label{globTerm}
	f_k(t) = \frac{\int_{M_t}{(kH_k + a^2 (n-k+2) H_{k-2}) F\, \mathrm{d\mu_t}}}{\int_{M_t}{(kH_k + a^2 (n-k+2)H_{k-2})\,\mathrm{d\mu_t}}}.
\end{equation}
Here $H_k$, $k = 0, ..., n$,  denotes the $k$-th elementary symmetric polynomial,
\begin{equation}
H_k(\kappa_1, \cdots, \kappa_n) = \sum_{i_1<\cdots <i_k} \kappa_{i_1}\cdots\kappa_{i_k},\quad \kappa = (\kappa_i) \in \bbr^n, \, 1\leq k \leq n,
\end{equation} $H_0 = 1$ and $\mathrm{d\mu_t}$ is the volume element of $M_t$. For $k\leq 1$ we use the same definition as in \cite{McCoyMixedAreaGen}:
\begin{equation}
	f_k(t) = \frac{\int_{M_t}{H_k F\, \mathrm{d\mu_t}}}{\int_{M_t}{H_k}\,\mathrm{d\mu_t}}.
\end{equation}

We remind the definition of mixed volumes: 
For $k \in \{0, \ldots, n\}$ and a strictly convex hypersurface $M$ in $\bbh^{n+1}_{\frac{1}{a}}$ represented by a graph $u$ over a geodesic sphere, i.e. $M = $ graph $u_{|\bbs^n}$,  we have:
\begin{equation}
	V_{n+1-k}(M) = 
	\begin{cases}
	\int_{\bbs^n}{\int_0^{u(x)}{a^{-n}\sinh^n(a s) \,\mathrm{ds}}\,\mathrm{d\sigma_n}(x)} , & k = 0\\
	{n \choose (k-1)}^{-1}\int_{M}{H_{k-1}\, \mathrm{d\mu}}, & k = 1, \ldots, n,
	\end{cases}
\end{equation}
where $\mathrm{d\sigma_n}$ is the volume element of the sphere.

The possible curvature functions are divided into two classes:
\begin{Annahme}
\label{MainAssumption}
Let $\alpha \in [0, 1]$. Suppose $\tilde{F}$ is a smooth, symmetric function defined on $\Gamma$, where $\Gamma$ is the positive cone $\Gamma_+ := \{\kappa = (\kappa_i)\in \bbr^n: \kappa_i > 0 \,\, \forall \, i \in \{1, \ldots, n\}\}$. Set $\Gamma_{\alpha} := \{\kappa = (\kappa_i)\in \bbr^n: \kappa_i > \alpha \quad \forall \, i \in \{1, \ldots, n\}\}$. Let $\eta_\alpha:\Gamma_\alpha \rightarrow \Gamma_+$, $\kappa \mapsto \kappa - \alpha e$, where $e= (1, \ldots, 1)$. Then we define the curvature functions by $F := \tilde{F}\circ \eta$. Furthermore we need the following assumptions for the curvature function $\tilde{F}$:
\begin{itemize}
\item $\tilde{F}$ is positively homogeneous of degree $1$, i.e. $\forall\, \kappa \in \Gamma_+$, $\forall\, \lambda \in \bbr_+$: $\tilde{F}(\lambda \kappa) = \lambda \tilde F(\kappa)$.
\item $\tilde{F}$ is strictly increasing in each argument: $\forall\, i \in \{1, \ldots, n\}$, $\forall\, \kappa \in \Gamma_+$ there holds $\tilde F_i(\kappa)$ = $\frac{\partial \tilde F}{\partial \kappa_i}(\kappa) > 0$.
\item $\tilde F$ is positive, $\tilde F_{|\Gamma_+} > 0$, and $\tilde F$ is normalized, $\tilde F(1, \ldots, 1) = 1$.
\item Either:
\begin{enumerate}[(i)]
\item $\tilde F$ is convex.
\item $\tilde F$ is concave and inverse concave, i.e. $\tilde F_{-1}(\kappa_i) := - \tilde F(\kappa_i^{-1})$ is concave.
\end{enumerate}
\end{itemize} 
\end{Annahme}
The most important examples of convex curvature functions $\tilde F$ fulfilling these assumptions (apart from the normalization) are the mean curvature $H = \sum_{i=1}^n \kappa_i$, the length of the second fundamental form $|A| = \sqrt{\sum_{i=1}^n\kappa_i^2}$ and the completely symmetric functions $\gamma_k = \left(\sum_{|\alpha| = k} \kappa^\alpha\right)^{\frac{1}{k}}$, $1\leq k\leq n$. For a proof that these curvature functions are convex see \cite[p. 105]{Mitr}. Examples of curvature functions $\tilde F$ from the second class, namely the ones being concave and inverse concave, are $\left(\frac{H_k}{H_l}\right)^{\frac{1}{k-l}}$, $n\geq k> l \geq 0$, or the power means $\left(\sum_{i=1}^n \kappa_i^r \right)^{\frac{1}{r}}$ for $|r|\leq 1$ . We refer to \cite[Section 2]{AndrewsPinching} for a proof of this fact and for an account of the theory of this class of curvature functions.

Now we can state the main theorem:

\begin{Theorem}
\label{MainTheorem1}
Let $x_0$ be stated as earlier and suppose $F$ is a function satisfying the conditions in Assumption \ref{MainAssumption}. Then the flow \eqref{floweq} with $f = f_k$, $k\in \{0, \ldots, n\}$, has a unique, smooth solution $x$ existing for all times $0\leq t < \infty$, the flow hypersurfaces $M_t$ remain strictly convex by horospheres and the volume $V_{n+1-k}$ is preserved during the flow. Furthermore the flow converges exponentially for $t\rightarrow \infty$ to a geodesic sphere of the same volume $V_{n+1-k}$ as $M_0$.
\end{Theorem}

Finally, we want to name some of the works about volume preserving curvature flows in different ambient manifolds and discuss shortly the results obtained in this work.

Volume preserving curvature flows have been considered for various curvature functions in different settings. Roughly speaking, if one assumes a certain convexity assumption or pinching condition on the initial hypersurface and shows that this condition is preserved during the flow, then after proving a priori estimates the existence of the flow for all times $t \in [0,\infty)$ can be deduced. If the exponential convergence of a suitable quantity can be shown, then by using interpolation inequalities, the exponential convergence of the flow to a sphere or a geodesic sphere in the $C^\infty$-topology can be inferred.

In the case the ambient manifold is $\bbr^{n+1}$, volume preserving mean curvature flows have been previously considered by Gage for $n=1$ in \cite{Gage} and by Huisken for $n\geq 2$ in \cite{HuisVol}. In a series of papers, McCoy considered first the area preserving mean curvature flow in \cite{McCoyArea}, then the mixed volume preserving mean curvature flow in \cite{McCoyMixedArea} and later on extended the results to very general curvature functions in \cite{McCoyMixedAreaGen}. 

In 2007 Cabezas-Rivas and Miquel proved similar results for a volume preserving mean curvature flow in the hyperbolic space by assuming that the initial hypersurface is horosphere-convex, see \cite{CabMiqHyp}. Recently, Gerhardt has considered in \cite{GerhardtHyp} inverse curvature flows of compact, starshaped hypersurfaces in hyperbolic space and has obtained the convergence of these flows to a geodesic sphere after an appropriate rescaling.

There are few results on curvature flows of compact hypersurfaces in more general Riemannian manifolds: In \cite{HuisRiem}, Huisken has considered the mean curvature flow in Riemannian manifolds of bounded curvature, Andrews proved a similar result in Riemannian manifolds of bounded curvature but considering only very particular curvature functions. In our notation, he essentially allowed for $\tilde F$ being the harmonic mean curvature function and $\alpha = a$. Lately, Xu has considered in \cite{Xu} the harmonic mean curvature flow in Hadamard manifolds. Furthermore there are works by Gerhardt, see \cite[Chapter 3]{GerhCP}, where forced curvature flows in ambient manifolds of non-positive or constant curvatures are considered, and the convergence of the flow is shown on the assumption that suitable barriers exist. As for the volume preserving mean curvature flow, there is a paper by Alikakos and Freire, see \cite{AliFre}. They assume that the scalar curvature of the ambient space has nondegenerate critical points and the initial hypersurface of the flow is close enough to a geodesic sphere and prove the long time existence and convergence of the flow to a hypersurface of constant mean curvature. 

Our work is mainly motivated by the approaches in the papers \cite{CabMiqHyp} and \cite{McCoyMixedAreaGen}: 

Short time existence of the flow has been shown in \cite{McCoyMixedAreaGen}. For a detailed account of the short time existence and also for a proof of the uniqueness of the flow we refer to our work in \cite{MakVolLor}. Hence we can suppose the flow exists in a maximal time interval $[0, T^*)$ for some $T^*>0$ and is smooth. Our approach is as follows:

First we show that the convexity by horospheres is preserved during the flow. Then we prove a pinching condition for the principal curvatures of the evolving hypersurfaces. This allows us to show that there exists $\eps>0$ such that the following holds: Let  $0\leq t_0<T^*-\eps$ and let $p_{t_0}$ be the center of an inball of $M_{t_0}$ (an inball of $M_{t_0}$ is a ball contained in the interior of $M_{t_0}$ with maximal radius). Now let us represent $M_{t_0}$ as a graph over the geodesic sphere with center $p_{t_0}$. Then for $t \in [t_0, t_0+\eps)$ the hypersurface $M_t$ can still be represented as a graph over this sphere, which is a consequence of the pinching estimate. With this result we can use a well known approach (also as in \cite{CabMiqHyp}) to estimate the curvature function $F$ from above. An application of the Harnack inequality yields the boundedness of $F$ from below. These estimates yield the existence for all times. Next, we use a new argument to show that the pinching of the principal curvatures improves at an exponential rate. This allows us to use an argument from Schulze in \cite{SchulzeConvexity} to obtain the exponential convergence of the flow to a geodesic sphere.

The paper is organized as follows. In section 2, we introduce the notation which is used throughout the paper and some results and inequalities concerning the curvature functions are stated. In section 3, we provide some facts about graphs over geodesic spheres in hyperbolic space and list the evolution equations of several important quantities. In section 4, we show that the mixed volume $V_{n+1-k}$ and an initial pinching of the principal curvatures of the hypersurfaces are preserved during the flow \eqref{floweq} with $f=f_k$. In section 5, we show that a graph representation is valid for a short but fixed time interval and prove the uniform boundedness of $F$. Section 6 treats the lower bound for $F$, which we infer from the Harnack inequality. The estimates obtained so far will then allow to conclude, that the flow exists for all times. In section 7 we prove that the flow converges exponentially to a geodesic sphere. Finally, Section 8 gives an example of how the flow can be used to deduce volume inequalities.

\section{Notation and Curvature functions}
\label{Notation}

The main objective of this section is to formulate the governing equations of a hypersurface in $\bbh^{n+1}_{\frac{1}{a}}$ and to provide some results about curvature functions. For more detailed definitions about curvature functions, we refer the reader to \cite[Chapter 2.1, 2.2]{GerhCP}. Unless stated otherwise, the summation convention is used throughout the paper.

We will denote geometric quantities in the ambient space $\bbh^{n+1}_{\frac{1}{a}}$ by greek indices with range from $0$ to $n$ and usually with a bar on top of them, for example the metric and the Riemannian curvature tensor in $\bbh^{n+1}_{\frac{1}{a}}$ will be denoted by $(\bar{g}_{\alpha\beta})$ and $(\bar{R}_{\alpha\beta\gamma\delta})$ respectively, etc., and geometric quantities of a hypersurface $M$ by latin indices ranging from $1$ to $n$, i.e. the induced metric and the Riemannian curvature tensor on M are denoted by $(g_{ij})$ and $(R_{ijkl})$ respectively. Ordinary partial differentiation will be denoted by a comma whereas covariant differentiation will be indicated by indices or in case of possible ambiguity they will be preceded by a semicolon, i.e. for a function $u$ in $\bbh^{n+1}_{\frac{1}{a}}$, $(u_\alpha)$ denotes the gradient and $(u_{\alpha\beta})$ the Hessian, but e.g. the covariant derivative of the curvature tensor will be denoted by $(\bar{R}_{\alpha\beta\gamma\delta;\epsilon})$. 

The induced metric of the hypersurface will be denoted by $g_{ij}$, i.e.
\begin{equation}
	g_{ij} = \langle x_i, x_j \rangle \equiv \bar{g}_{\alpha \beta} x_i^\alpha x_j^\beta,
\end{equation}
$(g^{ij})$ denotes the inverse of $(g_{ij})$, the second fundamental form will be denoted by $(h_{ij})$. The outer normal is denoted by $\nu$, i.e. if $M$ is a starshaped hypersurface represented as a graph in geodesic polar coordinates around a sphere with center in the interior of $M$, then we choose the normal $\nu$ such that there holds
\begin{equation}
	\langle \frac{\partial}{\partial r}, \nu\rangle > 0.
\end{equation}

The geometric quantities of the spacelike hypersurface $M$ are connected through the \begin{em}Gauß formula\end{em}, which can be considered as the definition of the second fundamental form,
\begin{equation}
	x_{ij} = -h_{ij}\nu.
\end{equation}

Note that here and in the sequel a covariant derivative is always a full tensor, i.e.
\begin{equation}
	x_{ij}^\alpha = x_{,ij}^\alpha -\Gamma_{ij}^k x_k^\alpha + \bar{\Gamma}_{\beta\gamma}^\alpha x_i^\beta x_j^\gamma,
\end{equation}
where $\bar{\Gamma}^\alpha_{\beta \gamma}$ and $\Gamma^k_{ij}$ denote the Christoffel-symbols of the ambient space and hypersurface respectively.

The second equation is the \begin{em}Weingarten equation\end{em}:
\begin{equation}
	\nu_i = h_i^k x_k = g^{kj} h_{ij} x_k.
\end{equation}

Finally, we have the \begin{em}Codazzi equation\end{em}
\begin{equation}
	h_{ij;k} = h_{ik;j} + \bar{R}_{\alpha\beta\gamma\delta}\nu^\alpha x_i^\beta x_j^\gamma x_k^\delta = h_{ik;j},
\end{equation}
as well as the \begin{em}Gauß equation\end{em}
\begin{equation}
\begin{split}
	R_{ijkl} &= \{h_{ik}h_{jl} - h_{il}h_{jk}\} + \bar{R}_{\alpha\beta\gamma\delta}x^\alpha_i x_j^\beta x_k^\gamma x_l^\delta\\
	& = \{h_{ik}h_{jl} - h_{il}h_{jk}\} + a^2\{g_{il}g_{jk} - g_{ik}g_{jl}\}.
\end{split}
\end{equation}

Now we want to give some facts about the curvature functions. Firstly, we provide the definition of these functions and mention some identifications, which will be used in the sequel without explicitly stating them again.
\begin{Definition}
	Let $\Gamma \subset \bbr^n$ be an open, convex, symmetric cone, i.e.
	\begin{equation}
		(\kappa_i) \in \Gamma \Longrightarrow (\kappa_{\pi i}) \in \Gamma \quad \forall \, \pi \in \mathcal{P}_n,
	\end{equation}
	where $\mathcal{P}_n$ is the set of all permutations of order $n$. Let $f \in C^{m, \alpha}(\Gamma)$, $m \in \bbn$, $0\leq \alpha \leq 1$, be \textit{symmetric}, i.e.,
	\begin{equation}
		f(\kappa_i) = f(\kappa_{\pi i}) \quad \forall \, \pi \in \mathcal{P}_n.
	\end{equation}
	Then $f$ is said to be a \textit{curvature function} of class $C^{m,\alpha}$. For simplicity we will also refer to the pair $(f, \Gamma)$ as a curvature function.
\end{Definition}
Now denote by $\mathbf{S}$ the symmetric endomorphisms of $\bbr^n$ and by $\mathbf{S}_\Gamma$ the symmetric endomorphisms with eigenvalues belonging to $\Gamma$, an open subset of $\mathbf{S}$. If $(f, \Gamma)$ is a smooth curvature function, we can define a mapping
\begin{equation}
\begin{split}
	\tilde{F}: &\,\mathbf{S}_\Gamma \rightarrow \bbr,\\
	&A\mapsto f(\kappa_i),
\end{split}
\end{equation}
where the $\kappa_i$ denote the eigenvalues of $A$. For the relation between these different notions, especially the differentiability properties and the relation between their derivatives, see \cite[Chapter 2.1]{GerhCP}. Since the differentiability properties are the same for $f$ as for $F$ in our setting, see \cite[Theorem 2.1.20]{GerhCP}, we do not distinguish between these notions and write always $F$ for the curvature function. Hence at a point $x$ of a hypersurface we can consider a curvature function $F$ as a function defined on a cone $\Gamma \subset \bbr^n$, $\tilde{F} = \tilde{F}(\kappa_i)$ for $(\kappa_i) \in \Gamma$ (representing the principal curvatures at the point $x$ of the hypersurface), as a function depending on $(h_i^j)$, $\tilde{F} = \tilde{F}(h_i^j)$ or as a function depending on $(h_{ij})$ and $(g_{ij})$, $\tilde{F} = \tilde{F}(h_{ij}, g_{ij})$. However, we distinguish between the derivatives with respect to $\Gamma$ or $\mathbf{S}$. We summarize briefly our notation and important properties:

	For a smooth curvature function $\tilde{F}$ we denote by $\tilde{F}^{ij} = \frac{\partial \tilde{F}}{\partial h_{ij}}$, a contravariant tensor of order 2, and $\tilde{F}^j_i = \frac{\partial \tilde{F}}{\partial h_j^i}$, a mixed tensor, contravariant with respect to the index $j$ and covariant with respect to $i$. We also distinguish the partial derivative $\tilde{F}_{,i} = \frac{\partial \tilde{F}}{\partial \kappa_i}$ and the covariant derivative $\tilde{F}_{;i} = \tilde{F}^{kl}h_{kl;i}$. Furthermore $\tilde{F}^{ij}$ is diagonal if $h_{ij}$ is diagonal and in such a coordinate system there holds $\tilde{F}^{ii} = \frac{\partial \tilde{F}}{\partial \kappa_i}$. For a relation between the second derivatives see \cite[Lemma 2.1.14]{GerhCP}. Finally, if $\tilde{F} \in C^2(\Gamma)$ is concave (convex), then $\tilde{F}$ is also concave (convex) as a curvature function depending on $(h_{ij})$. 
	
For $\alpha \in [0,1]$ and $\eta_\alpha$ as in the assumption \ref{MainAssumption} we can treat the derivatives of $F= \tilde{F}\circ\eta_\alpha$ essentially as above by using the chain rule.

With these definitions we can turn to special classes of curvature functions. 

We note some important properties of the elementary symmetric polynomials:
\begin{Lemma}
\label{symPol}
Let $1\leq k \leq n$ be fixed.
	\begin{enumerate}[(i)]
		\item We define the convex cone 
			\begin{equation}
			\label{GammaK}
				\Gamma_k = \{(\kappa_i) \in \bbr^n: H_1(\kappa_i) > 0, H_2(\kappa_i) > 0, \ldots, H_k(\kappa_i) > 0 \}.
			\end{equation}
			Then $H_k$ is strictly monotone on $\Gamma_k$ and $\Gamma_k$ is exactly the connected component of
			\begin{equation}
				\{(\kappa_i) \in \bbr^n: H_k(\kappa_i) > 0\}
			\end{equation}
			containing the positive cone.
		\item For fixed $i$, no summation over $i$, there holds
			\begin{equation}
			\label{DerElemPol}
				H_k = \frac{\partial H_{k+1}}{\partial \kappa_i} + \kappa_i \frac{\partial H_k}{\partial \kappa_i}.
			\end{equation}
	\end{enumerate}
\end{Lemma}
\begin{proof}
	The convexity of the cone $\Gamma_k$ and (i) follows from \cite[Section 2]{HuiskSinestr}
	and (ii) follows directly from the definition of the $H_k$.
\end{proof}\qed

As a consequence we obtain:
\begin{Lemma}
\label{divFree}
	Let $N$ be a semi-Riemannian space of constant curvature, then for the symmetric polynomials $F= H_k$, $1\leq k \leq n$, the tensor $F^{ij}$ evaluated at $M$, where $M$ is an arbitrary admissible hypersurface, is divergence free. In case $k=2$ it suffices to assume that $N$ is an Einstein manifold.
\end{Lemma}
\begin{proof}
	The proof of the Lemma can be found in \cite[Lemma 5.8]{GerhSurvey}.
\end{proof}
Now we state some well-known facts for general curvature functions:
\begin{Lemma}
There holds:
\begin{enumerate}[(i)]
\item
	Let $F \in C^2(\Gamma_+)$ be a concave (convex) curvature function, homogeneous of degree 1 with $F(1,\ldots, 1) > 0$, then 
	\begin{equation}
	\label{FHIneq}
	F\leq (\geq) \frac{F(1, \ldots, 1)}{n} H.
	\end{equation}
\item Let $F\in C^2(\Gamma_+)$ be a strictly monotone, concave (convex) curvature function, positively homogeneous of degree 1, then for all $\kappa \in \Gamma_+$ there holds
\begin{equation}
\label{FGIneq}
\sum_{i=1}^nF_i(\kappa) \geq (\leq)\, F(1, \ldots, 1).
\end{equation}
\item Let $F\in C^2(\Gamma_+)$ be a curvature function. Then if $F$ is convex (concave) in $\Gamma_+$, then at all $\kappa\in \Gamma_+$ we have for all $i\neq j$
	\begin{equation}
	\label{FConvex}
		\frac{F_i - F_j}{\kappa_i - \kappa_j} \geq (\leq)\, 0.
	\end{equation}
\end{enumerate}
\end{Lemma}
\begin{proof}
	See \cite[Lemma 2.2.20, Lemma 2.2.19, Lemma 2.1.14]{GerhCP}.
\end{proof}

\section{Graph representation, evolution equations}

First of all, we cite Hadamard's theorem in hyperbolic space, for a proof see \cite[Theorem 10.3.1]{GerhCP}. Since the proof can be easily adjusted to the hyperbolic space of radius $a^{-1}$, we only state the result:
\begin{Theorem}
Let $M$ be a compact, connected, $n$-dimensional manifold and
\begin{equation}
x:M\rightarrow \bbh^{n+1}_{\frac{1}{a}}
\end{equation}
a strictly convex immersion of class $C^2$, i.e., the second fundamental form with respect to any normal is always (locally) invertible. Then the immersion is actually an embedding and $\tilde{M} = x(M)$ is a strictly convex hypersurface that bounds a strictly convex body $\hat M \subset \bbh^{n+1}$. $\tilde M$ and $M$ are moreover diffeomorphic to $\bbs^n$ and orientable.
\end{Theorem}

The fact that such a hypersurface bounds a strictly convex body makes it possible to represent it as a graph over a geodesic sphere. 
Hence let $M$ be a strictly convex hypersurface in $\bbh^{n+1}_{\frac{1}{a}}$, let $p \in $ int $\hat M$ and consider geodesic polar coordinates centered at $p$. Then the metric can be expressed as
\begin{equation}
 d\bar{s}^2 = dr^2 + \bar{g}_{ij} dx^idx^j,
\end{equation}
where $\sigma_{ij}$ is the canonical metric of $\bbs^n$ and 
\begin{equation}
\label{Metlvl}
\bar{g}_{ij} = a^{-2} \sinh^2 (a r) \sigma_{ij}
\end{equation} is the induced metric of $S_r(p)$, the geodesic spheres with center $p$ and radius $r$. A simple calculation using $\bar{h}_{ij} = \frac{1}{2} \dot{\bar{g}}_{ij}$ yields 
\begin{equation}
\label{SecFFlvl}
\bar{h}_{ij} = a\coth(a r) \bar{g}_{ij},
\end{equation} 
where $\bar{h}_{ij}$ denotes the second fundamental form of $S_r(p)$.

Let $M = $ graph $u_{|\bbs^n} = \{(x^0, x): x^0 = u(x), x\in \bbs^n\}$, then the induced metric has the form
\begin{equation}
	g_{ij} = u_iu_j + \bar{g}_{ij},
\end{equation} 
where $\bar{g}_{ij}$ is evaluated at $(u,x)$ and its inverse $(g^{ij}) = (g_{ij})^{-1}$ can be expressed as
\begin{equation}
	g^{ij} = \bar{g}^{ij} - v^{-2} u^i u^j,
\end{equation}
where $u^i = \bar{g}^{ik} u_k$ with $(\bar{g}^{ik}) = (\bar{g}_{ik})^{-1}$ and 
\begin{equation}
v^2 = 1 + \bar{g}^{ij} u_iu_j \equiv 1 + |Du|^2.
\end{equation}  
The outward normal has the following representation in these coordinates
\begin{equation}
(\nu^\alpha) = v^{-1}(1, -u^i).
\end{equation}
Looking at the component $\alpha=0$ in the Gau\ss{} formula yields the equation
\begin{equation}
v^{-1}h_{ij} = -u_{ij} - \bar{\Gamma}^0_{00} u_i u_j - \bar{\Gamma}^0_{0i} u_j -\bar{\Gamma}^0_{0j} u_j - \bar{\Gamma}^0_{ij},
\end{equation}
where the covariant derivatives are taken with respect to the induced metric of $M$ and
\begin{equation}
- \bar{\Gamma}^0_{ij} = \bar{h}_{ij}, \quad \bar{\Gamma}^0_{00} = \bar{\Gamma}^0_{0i}  = \bar{\Gamma}^0_{0j} = 0.
\end{equation}

From now on we fix $k$, $0\leq k\leq n$, and consider the flow \eqref{floweq} with $f=f_k$.
As already mentioned in the introduction, short-time existence has been proved for this flow, so we can assert the flow exists in the class $C^{\infty}$ in the time interval $[0, T^*)$ for some $T^*>0$. Hence we can state the evolution equations of the quantities to be used in the sequel, where we note that all derivatives are covariant derivatives taken with respect to the induced metric of $M$ and the time derivatives are total derivatives, i.e. covariant derivatives of tensor fields defined over the curve $x(t)$. Let $g := \det(g_{ij})$ and note that $\bar{H}n^{-1} = a\cosh(au)(\sinh(au))^{-1}$. Furthermore, for a function $\varphi \in C^{\infty}(M)$ we define $L \varphi := \dot\varphi - F^{kl}\varphi_{;kl}$, with a similar definition for tensors.
\begin{Lemma} (Evolution equations)
\allowdisplaybreaks{
\begin{align}
& \dot{g}_{ij} = -2(F-f) h_{ij},\\
\label{EvDet}
& \dot{\sqrt{g}} = -(F-f)H \sqrt{g},\\
\label{dotU}
& \dot u = -v^{-1}(F - f)\\
\label{partialU}
& \frac{\partial u}{\partial t} = -v(F-f)\\
\label{EvGraph}
& Lu = (\alpha F^{ij} g_{ij} + f) v^{-1} + (-F^{ij}g_{ij} + F^{ij}u_iu_j) \bar{H}n^{-1},\\
& L\chi = - F^{ij}h^k_ih_{kj} \chi- 2 \chi^{-1} F^{ij} \chi_i\chi_j + (2 F +\alpha F^{ij}g_{ij} - f)\bar{H}n^{-1} v\chi,\\
&\textnormal{ where } \chi = \frac{v}{\sinh(a u)},\notag\\
\label{EvF}
& LF = (F^{ij}h^k_ih_{kj} - a^2 F^{ij}g_{ij}) (F-f),\\
\label{Doth}
& \dot{h}^j_i = (F-f)^j_i + (F-f) h^k_ih_k^j - (F -f) a^2 \delta_i^j\\
&L h^i_j = (F^{kl}h_{rk}h^r_l + a^2 F^{kl}g_{kl}) h^i_j - (\alpha F^{kl}g_{kl} + f) h^k_jh^i_k \\
&\quad- a^2 (2F + \alpha F^{ij}g_{ij} -f) \delta^i_j + F^{kl,rs}h_{kl;j} h_{rs;}^{\quad i}\\
\label{EvSecFF}
&L h_{ij} = (F^{kl}h_{rk}h^r_l + a^2 F^{kl}g_{kl})h_{ij} - (\alpha F^{kl}g_{kl} + f) h^k_ih_{kj} \\
&\quad + F^{kl,rs}h_{kl;i}h_{rs;j} - a^2 (2 F + \alpha F^{ij}g_{ij} - f) g_{ij} - 2(F-f) h^k_ih_{kj}.
\end{align}
}
\end{Lemma}
\begin{proof}
For a proof see \cite[Chapter 2]{GerhCP}. Note that the curvature functions are homogeneous of degree 1 in $\kappa_i - \alpha$, hence we have $F^{ij}h_{ij} = F + \alpha F^{ij}g_{ij}$. This has to be taken into account for a derivation of the evolution equations.
\end{proof}

\section{Preserved quantities}
In this section we show which quantities are preserved during the flow.

First of all we show that the mixed volume $V_{n+1-k}$ is preserved:
\begin{Lemma}
\label{PresVol}
	The mixed volume $V_{n+1-k}$ is preserved during the flow, i.e. $V_{n+1-k}(M_t) = V_{n+1-k}(M_0)$ for all $t\in [0, T^*)$.
\end{Lemma}
\begin{proof}
	\begin{enumerate}[(i)]
	\item $k=0$:
	First we observe that for $x\in \bbs^n$ we have 
	\begin{equation}
		\sqrt{g(u(x), x)} = v\sqrt{\det(\bar{g}_{ij}(u(x), x))}.
	\end{equation}
	
	Taking this into account, we have for $k=0$ in view of \eqref{partialU}:
	\begin{equation}
	\begin{split}
		\frac{d}{dt} V_{n+1} &= \int_{\bbs^n}{\frac{\partial u}{\partial t}\,\frac{\sqrt{\bar{g}(u(x), x)}}{\sqrt{\sigma(x)}} \, \mathrm{d\sigma_n(x)}}\\
		&= -\int_{\bbs^n}{(F - f_0) \frac{\sqrt{g(u(x),x)}}{\sqrt{\sigma(x)}}\,\mathrm{d\sigma_n(x)}} = 0,
	\end{split}
	\end{equation}
	in view of the definition of $f_0$. Hence the enclosed volume is preserved by the flow.
	
	\item $k=1$: We have in view of \eqref{EvDet}
	\begin{equation}
	\begin{split}
		n \frac{d}{dt} V_n = \frac{d}{dt} |M_t|
		= -\int_{M_t}{(F - f_1) H\, \mathrm{d\mu_t}} = 0.
	\end{split}
	\end{equation}
	
	\item $1 < k \leq n$: Now we exploit Lemma \ref{divFree} and the identity \eqref{DerElemPol}.
	We get from \eqref{Doth} and \eqref{EvDet}
	\begin{equation}
	\begin{split}
	{n \choose k}&\frac{d}{dt} \int_{M_t}{H_{k-1} \mathrm{d\mu_t}} = \int_{M_t}{(F-f_k)(H_{k-1})^i_jh^j_kh_i^k \mathrm{d\mu_t}} \\
	&- a^2 \int_{M_t}{(F- f_k)(H_{k-1})^i_j\delta_i^j \mathrm{d\mu_t}} -\int_{M_t} (F-f_k) H_{k-1}H \mathrm{d\mu_t}\\
	&= - \int_{M_t}{(F - f_k) \left\{k H_k + a^2 (n-k+2) H_{k-2}\right\}\,\mathrm{d\mu_t}} = 0.
	\end{split}
	\end{equation}
	\end{enumerate}
\end{proof}

Next, we want to prove that a pinching of the principal curvatures of the initial hypersurface is preserved during the flow. First we cite a modification of Hamiltons maximum principle for tensors by Andrews, see \cite[Theorem 3.2]{AndrewsPinching}:

\begin{Theorem}
\label{AndPinch}
	Let $S_{ij}$ be a smooth time-varying symmetric tensor field on a compact manifold $M$ (possibly with boundary), satisfying
	\begin{equation}
		\dot{S}_{ij} = a^{kl}S_{ij;kl} + u^kS_{ij;k} + N_{ij},
	\end{equation}
	where $a^{kl}$ and $u^k$ are smooth and the covariant derivatives are taken with respect to a smooth, possibly time-dependent, symmetric connection and $a^{kl}$ is positive definite everywhere. Suppose that
	\begin{equation}
		N_{ij}v^iv^j + \sup_{\Gamma} 2a^{kl}(2\Gamma^p_kS_{ip;l}v^i - \Gamma^p_k\Gamma^q_l S_{pq}) \geq 0,
	\end{equation}
	whenever $S_{ij} \geq 0$ and $S_{ij}v^j =0$. If $S_{ij}$ is positive definite everywhere on $M$ at time $t=0$ and on $\partial M$ for $0\leq t \leq T$, then it is positive definite on $M\times [0,T]$.
\end{Theorem}
In the paper cited above, roughly said, Andrews uses this maximum principle to derive that a certain curvature pinching for closed hypersurfaces in $\bbr^{n+1}$ is preserved for curvature functions that are both concave and inverse concave (and satisfy the other conditions of assumption \ref{MainAssumption} apart from the convexity). To do so, he needs another important Theorem, which holds for such curvature functions, namely \cite[Theorem 4.1]{AndrewsPinching}. We need a slightly generalized version of this Theorem to apply Theorem \ref{AndPinch} to obtain the preservation of a curvature pinching in our situation. The proof is identical to the one of \cite[Theorem 4.1]{AndrewsPinching}, we only need a minor observation at the beginning of the proof.

\begin{Theorem}
\label{AndPinchLem}
	Let $\alpha \in \bbr_+$. Let $F$ be a smooth, symmetric, monotone, concave and inverse-concave curvature function defined on 
	\begin{equation}
		\Gamma_\alpha := \{\lambda = (\lambda_i) \in \bbr^n: \lambda_k > \alpha \,\forall\, k \in \{1, \ldots, n\}\}.
	\end{equation}	
	Let $A$ be a symmetric 2-Tensor with eigenvalues in $\Gamma_\alpha$ and $v$ an eigenvector of $A$ corresponding to the smallest eigenvalue of $A$. Let $\tilde{A} := A - \alpha I$, where $I$ is the identity matrix, and let $\eps := \frac{\tilde{A}_{ij}v^iv^j}{\Tr \tilde{A} |v|^2} \in (0, \frac{1}{n})$. If $T$ is a totally symmetric $3$-tensor with $T_{ijk}v^iv^j = \eps\,\delta^{ij}T_{ijk}$ for ${k=1, \ldots, n}$, then
	\begin{equation}
	\begin{split}
		\beta &:= v^iv^jF^{kl,pq}(A) T_{ikl}T_{jpq} - \eps |v|^2 \delta^{ab}F^{kl,pq}(A)T_{akl}T_{bpq}\\
		&+ 2 \sup_{\Gamma} F^{kl}(A)\left( 2 \Gamma^p_k(T_{lpi}v^i - \eps\delta^{ab}T_{lab}v_p) - \Gamma^p_k\Gamma^q_l(\tilde{A}_{pq} - \eps \Tr \tilde{A} \delta_{pq}\right) \geq 0.
	\end{split}
	\end{equation}
\end{Theorem}
\begin{proof}
	Firstly, we note that \cite[Corollary 5.5]{AndrewsPinching} remains valid if $\Omega = \Gamma_\alpha$, hence the inverse of the curvature function is concave as a function of the principal curvatures if and only if it is concave as a function of the second fundamental form.
	For fixed $v$ and $T$ the quantity $\beta$ is upper semi-continuous in $A$, since $F$ is smooth. This allows us to assume that all eigenvalues of $A$ are distinct, since otherwise we can take a sequence $\{A^{(k)}\}_{k\geq 0}$ with $A^{(k)} \rightarrow A$ for $k\rightarrow \infty$, $\tilde{A}_{ij}{(k)} := A^{(k)}_{ij} - \alpha \delta_{ij}$, $\tilde{A}^{(k)}_{ij} \geq \eps \Tr \tilde{A}^{(k)} \delta_{ij}$ and $\tilde{A}^{(k)}_{ij}v^iv^j = \eps \Tr \tilde{A}^{(k)} |v|^2$, such that the eigenvalues of each $A^{(k)}$ are distinct. 
	
	Let us take an orthonormal basis $e_1, \ldots, e_n$ of eigenfunctions of $A$, with eigenvalues in increasing order. Then we have in this basis $v=e_1$, $A= \diag(\lambda_1, \ldots, \lambda_n)$, $F^{ij} = f^i\delta^{ij}$ and $\lambda_1 = \alpha - n\alpha\eps + \eps H$. We observe that the last identity implies 
	\begin{equation}
		\tilde{A}_{pq} - \eps \Tr \tilde{A}\delta_{pq} =  \lambda_p\delta_{pq} - \lambda_1\delta_{pq}.
	\end{equation}
	
	Now we can explicitly determine the $\Gamma$ at which the supremum in $\beta$ is attained. 
	\begin{equation}
	\begin{split}
		2ÊF^{kl}&\left((2 \Gamma_k^p(T_{kpi}v^i - \eps T_{laa}v_p) - \Gamma_k^p\Gamma_l^q(\tilde{A}_{pq} - \eps \Tr \tilde{A} \delta_{pq})\right)\\
		&= 2 \sum_{k=1}^n\sum_{p=2}^n f^k\left(2\Gamma_k^pT_{kp1} - (\Gamma_k^p)^2 (\lambda_p - \lambda_1)\right))\\
		&= 2\sum_{k\geq 1, p\geq 2}\left(\frac{f^k}{\lambda_p-\lambda_1}- f^k(\lambda_p - \lambda_1)\left(\Gamma_k^p - \frac{T_{kp1}}{\lambda_p - \lambda_1}\right)^2\right).
	\end{split}
	\end{equation}
	Hence it follows, that the supremum is attained by the choice $\Gamma^p_k := \frac{T_{kp1}}{\lambda_p - \lambda_1}$. At this point we are in the exact same situation as in \cite[Theorem 4.1]{AndrewsPinching} and the rest of the proof remains the same.
\end{proof}

The preceding two Theorems allow us to prove the pinching estimate for our flow:

\begin{Lemma}
\label{PresPinch}
Let $\eps>0$ be a constant such that we have $\kappa_1 - a\geq \tilde{\eps} (H- n a)$ for all $x\in M_0$ in the case of concave curvature functions and $\kappa_1 - a\geq \eps (F - (a-\alpha))$ for all $x\in M_0$ in the case of convex curvature functions, where $\kappa_1$ denotes the smallest principal curvature of $M_0$ at $x$. Let us assume that the hypersurfaces $M_t$ remain strictly $h$-convex for $t\in [0, T^*)$. Let $\epst = \frac\eps n$. Then for every $t\in [0, T^*)$ and $x \in M_t$ there holds
\begin{align}
& \kappa_1 - a \geq \eps(F-(a-\alpha)),\\
\label{PresPinchEq}
&\kappa_1 - a\geq \tilde \eps (H- na),
\end{align}
where $\kappa_1$ denotes the smallest principal curvature of $M_t$ at $x$.
\end{Lemma}
\begin{proof}
We need to distinguish between convex and concave curvature functions:
\begin{enumerate}[a)]
	\item Firstly, we assume $F$ to be a convex curvature function.
		
	Let $S_{ij} = h_{ij} - (a + \eps\, (F - (a-\alpha))) g_{ij}$. Then we have
	\begin{equation}
	\begin{split}
		\dot{S}_{ij} &- F^{kl}S_{ij;kl} = (F^{kl}h^r_kh_{rl} + a^2 F^{kl}g_{kl})h_{ij} - (f + \alpha F^{kl}g_{kl})  h_i^rh_{rj}\\
		& - a^2 (2F + \alpha F^{kl}g_{kl} - f) g_{ij} +F^{kl,rs}h_{kl;i}h_{rs;j} - 2 ( F - f ) h^k_i S_{kj} \\
		&- \eps(F^{kl}h^r_kh_{rl} - a^2 F^{kl}g_{kl})(F-f) g_{ij}. 
	\end{split}
	\end{equation}
	Denote the right hand side by $N_{ij}$. Let $0<t_0<T^*$ and $x_0 \in M_{t_0}$ be such that at $x_0$ there holds $S_{ij}\geq 0$ and there exists a normalized null eigenvector $v$ for $(S_{ij})$, i.e. $S_{ij}v^j = 0$ and $|v|^2 =1$. We introduce Riemannian normal coordinates at $x_0$ such that the principal curvatures at $x_0$ are monotonically ordered, $\kappa_1 \leq \kappa_2 \leq \ldots \leq \kappa_n$. Note that we have $\kappa_1 = h_{ij} v^iv^j$. At $x_0$ there holds due to the convexity and homogeneity of $F$
	\begin{equation}
	\begin{split}
		N_{ij}v^i v^j &\geq F^{kl}h^r_kh_{rl} (h_{ij}v^iv^j - \eps F ) + a^2 F^{kl}g_{kl} (h_{ij}v^iv^j + \eps FÊ) - \alpha a^2 F^{kl}g_{kl} \\
		& - 2 a^2 F - \alpha F^{kl}g_{kl} \kappa_1^2 + f \left(-\kappa_1^2 + (\eps F^{kl}h^r_kh_{rl} - a^2 \eps F^{kl}g_{kl} + a^2) \right)\\
		&= ((1-\eps) a + \eps \alpha) F^{kl}h^r_kh_{rl} + a^2 F^{kl}g_{kl} (2\eps F + (1-\eps) a + \eps \alpha)-2 a^2 F^{kl}h_{kl} \\
		& + \alpha (a^2-\kappa_1^2) F^{kl}g_{kl} + f \left[ -\kappa_1^2 + \eps F^{kl}h^r_kh_{rl} + a^2 (1-\eps F^{kl}g_{kl})\right].
	\end{split}
	\end{equation}
	The part in the square brackets is non-negative:
	\begin{equation}
	\begin{split}
		& a^2-\kappa_1^2 + \eps \sum_{i=1}^n f_i (\kappa_i^2 - a^2) \\
		&\geq -\eps (F - (a-\alpha)) (a+\kappa_1) + \eps (a+\kappa_1) \sum_{i=1}^n f_i (\kappa_i - \alpha + \alpha - a)\\
		& = \eps (a+\kappa_1)\left(-F + a - \alpha + F - \sum_{i=1}^nf_i (a-\alpha)\right) \geq 0,
	\end{split}
	\end{equation}	
	since we have $\sum_{i=1}^nf_i \leq 1$ in view of inequality \eqref{FGIneq}.
	A short computation also yields
	\begin{equation}
		a^2-\kappa_1^2 = -2\eps a(F-(a+\alpha)) - \eps^2 (F- (a-\alpha))^2.
	\end{equation}
	Since $f\geq a-\alpha$, we obtain
	\begin{equation}
	\begin{split}
		N_{ij}&v^iv^j \geq a \sum_{i=1}^n f_i\kappa_i^2 - 2 a^2 \sum_{i=1}^nf_i\kappa_i + a^3 \sum_{i=1}^n f_i\\
		&+ (a^2-\kappa_1^2)(a-\alpha + \alpha\sum_{i=1}^n f_i) + 2a^2\eps(F-(a-\alpha))  \sum_{i=1}^n f_i \\
		&= a\sum_{i=1}^nf_i(\kappa_i - a)^2 - \eps^2 (F- (a-\alpha))^2 \left(a - \alpha\left(1-\sum_{i=1}^n f_i\right)\right)\\
		& - 2\eps a (a-\alpha) (F-(a-\alpha)) \left(1- \sum_{i=1}^nf_i\right).
	\end{split}
	\end{equation}
	With
	\begin{equation}
	\sum_{i=1}^n f_i (\kappa_i - a)^2 = \sum_{i=1}^n f_i(\kappa_i-\kappa_1)^2 + 2\sum_{i=1}^n f_i(\kappa_i-\kappa_1)(\kappa_1-a) + \sum_{i=1}^n f_i (\kappa_1 - a)^2
	\end{equation}
	and
	\begin{equation}
	\begin{split}
		\sum_{i=1}^n& f_i (\kappa_i - \kappa_1) (\kappa_1 - a) = (\kappa_1 - a) \left(F - (\kappa_1-\alpha) \sum_{i=1}^n f_i\right)\\
		& = \eps(F - (a-\alpha)) \left((a-\alpha) \left(1-\sum_{i=1}^nf_i\right) \right.\\
		&\left.+ (1-\eps) (F-(a-\alpha)) + \eps (F-(a-\alpha)) \left(1 - \sum_{i=1}^nf_i\right) \right)
	\end{split}
	\end{equation}
	we derive
	\begin{equation}
	\label{PinchMainConv}
	\begin{split}
		N_{ij}&v^iv^j \geq 2a\eps (1-\eps) (F-(a-\alpha))^2 \geq 0.
	\end{split}
	\end{equation}	
	An application of Theorem \ref{AndPinch} and \eqref{FHIneq} finishes the proof in the case of convex curvature functions.

	\item	Next, we assume $F$ to be concave and inverse concave.
	
	Let $S_{ij} = h_{ij} + (-a - \epst H + \epst a\, n)g_{ij}$.
	
	Then $S_{ij}$ satisfies the following evolution equation:
	\begin{equation}
	\begin{split}
		\dot{S}_{ij} &- F^{kl}S_{ij;kl} = (F^{kl}h_{rk}h^r_l + a^2F^{kl}g_{kl}) (h_{ij} - \epst H g_{ij}) \\
	&+ (f + \alpha F^{kl}g_{kl}) (\epst|A|^2g_{ij} - h^k_ih_{kj}) + F^{kl,rs}h_{kl;p}h_{rs;q}(\delta^p_i\delta^q_j - \epst g^{pq}g_{ij})\\
	&- a^2 (2F + \alpha F^{kl}g_{kl} -f) (1- \epst n) g_{ij}
- 2(F-f) h^k_iS_{kj}.
	\end{split}
	\end{equation}
	We denote the right hand side of this equation by $N_{ij}$. Now we want to use Theorem \ref{AndPinch} to obtain \eqref{PresPinchEq}. 
	
	Let $0<t_0<T^*$ and $x_0 \in M_{t_0}$ be such that at $x_0$ there holds $S_{ij}\geq 0$ and there exists a normalized null eigenvector $v$ for $(S_{ij})$, i.e. $S_{ij}v^j = 0$ and $|v|^2 =1$. We introduce Riemannian normal coordinates at $x_0$ such that the principal curvatures at $x_0$ are monotonically ordered, $\kappa_1 \leq \kappa_2 \leq \ldots \leq \kappa_n$. Note that we have $\kappa_1 = h_{ij} v^iv^j$. Using Theorem \ref{AndPinchLem} we only need to show that the remaining terms in $N_{ij}v^iv^j$ are non-negative:
	\begin{equation}
	\begin{split}
		&(F^{kl}h_{rk}h^r_l + a^2 F^{kl}g_{kl}) (h_{ij} v^iv^j - \epst H) - a^2 (2F+\alpha F^{kl}g_{kl})\,(1-\epst n)\\ 
		&+\alpha F^{kl}g_{kl}(\epst|A|^2- \kappa_1^2)+ f(\epst |A|^2  - h^k_ih_{kj}v^iv^j + a^2 (1-\epst n))\\ 
		& -2(F-f)h_i^kS_{kj}v^iv^j\\
		&\geq (f+ \alpha F^{kl}g_{kl}) (\epst |A|^2 + a^2 (1-\epst n) - ((1-\epst n) a + \epst H)^2)\\
		& \quad + (F^{kl}h_{rk}h^r_l + a^2 F^{kl}g_{kl} - 2a \, F^{kl}h_{kl})\,a\,(1-\epst n).
	\end{split}
	\end{equation}
	The terms involving $(f+\alpha F^{kl}g_{kl})$ are positive as can be seen by using the binomial inequality and noting that the hypersurface is strictly convex at $x_0$; we only consider the terms in the brackets:
	\begin{equation}
	\label{estimatePinch}
	\begin{split}
		&\epst |A|^2 + a^2 (1-\epst n)- a^2 (1-\epst n)^2 - \epst^2 H^2 - 2 \epst a(1-\epst n)H \\
		&\geq \epst H^2\left(\frac{1}{n} - \epst - \frac{1-\epst n}{n}\right) + a^2 (1-\epst n) \left(1 - (1- \epst n) - \epst n\right) = 0.
	\end{split}
	\end{equation} 
	The remaining terms are positive, since they can be expressed as
	\begin{equation}
	\begin{split}
		a(1-\epst n)&\sum_{i=1}^n f_i (\kappa_i^2 + a^2 -2 a\kappa_i) = a(1-\epst n) \sum_i f_i \left(\kappa_i - a\right)^2\\
		&\geq a(1-\epst n) \epst^2 (H- an)^2 \sum_{i=1}^n f_i \geq 0.
	\end{split}
	\end{equation}
	Hence we obtain
	\begin{equation}
		\label{PinchMainConc}
		N_{ij}v^i v^j \geq \epst^2 (1-\epst n) a (H- an)^2 \sum_{i=1}^n f_i \geq 0.
	\end{equation}
\end{enumerate}
\end{proof}

Curvature pinching has an important consequence, which follows from the fact that our curvature functions are homogeneous of degree $1$ and hence the derivative of the curvature function is homogeneous of degree $0$:

\begin{Korollar}
\label{UniformlyParabolic}
There exists a constant $c_0> 0$ depending only on $n, M_0$ and the curvature function, such that for every $t \in [0, T^*)$ and $x \in M_t$ there holds 
\begin{equation}
c_0^{-1} \delta^i_j \leq F^i_j((h^k_l)(x)) \leq c_0 \delta^i_j
\end{equation} 
holds as long as the hypersurfaces $M_t$ are strictly $h$-convex.
\end{Korollar}
\begin{proof}
	We can argue exactly as in \cite[Corollary 4.6]{AndrewsContractionEuc}, only we define $\lambda_i := \kappa_i - \alpha$ for $i\in \{1, \ldots, n\}$ and $\lambda = (\lambda_i)$.
\end{proof}

\section{Estimates of the principal curvatures}

Throughout this section we will assume that the hypersurfaces remain strictly $h$-convex as long as the flow exists. We will justify this assumption in the next section.

We will see  that we can bound $F$ uniformly from above, if we have an upper bound on $\chi$. Hence our goal is to estimate $\chi$ from above for some small but fixed interval $[0, \eps]$, only depending on bounded quantities.

Firstly, we note the following:
\begin{Lemma}
\label{LowBoundGraph}
Let $t_0 \in [0, T^*)$ be fixed and let $M_{t_0}$ be a graph over the geodesic sphere with center equal to the center of the inball of $M_{t_0}$, $M_{t_0} = $ graph $u_{|\bbs^n}$. Choose $\beta >0 $ such that $e^{\beta} \leq \inf_{M_{t_0}} \cosh (au)$. 

Let $t_1 := \min\{t_0 + \frac{\beta}{2a^2c_0}, T^*\}$. Then for $t\in [t_0, t_1)$ the graph representation is still valid for $M_t$ and we have the estimate
\begin{equation}
\label{LowBoundGraphEq}
	u \geq \frac{\beta}{2a}.
\end{equation}
Furthermore we also get an upper estimate for $\chi$:
\begin{equation}
\label{BoundChiEq}
 \sup_{t\in [t_0, t_1)}\sup_{x\in M_t}\chi(x) \leq \sup_{t\in [t_0, t_1)}\sup_{x\in M_t}\frac{1}{\sinh(au)(x)} \leq \frac{1}{\sinh\left(\frac{\beta}{2}\right)}.
\end{equation}
\end{Lemma}

\begin{proof}
Define $\varphi:= e^{a^2c_0 (t-t_0)}\cosh(au)$. Let $0<T<T^*$. Let $x_0 = x_0(t_0)$, with $0<t_0\leq T$, be a point in $M_{t_0}$ such that  
\begin{equation}
	\sup_{M_0} \varphi < \sup\left\{\sup_{M_t} \varphi: 0<t\leq T\right\} = \varphi(x_0).
\end{equation}
In view of the maximum principle we obtain from \eqref{EvGraph}:
\begin{equation}
0 \geq \varphi^{-1}(\dot\varphi - F^{ij}\varphi_{;ij}) > a^2c_0 - F^{ij}g_{ij} a^2 \geq 0.
\end{equation}
Hence 
\begin{equation}
	\inf_{M_t} \cosh(au) \geq e^{-a^2c_0(t-t_0)}\inf_{M_0} \cosh(au),
\end{equation} 
which implies
\begin{equation}
	u\geq a^{-1}\left(\beta - a^2c_0 (t-t_0)\right).
\end{equation}
This proves the first part of the claims.

Let $\delta > 0$ be small. Let $t\in [t_0, t_1-\delta]$ and $x_0 \in M_t$ be given such that $\chi_{|[t_0, t_1-\delta]}$ assumes its supremum at $x_0$. Then we have $\chi_i = 0$ for $i\in \{1, \ldots n\}$, which is tantamount to
\begin{equation}
	0 = \frac{v_i}{\sinh(au)} - v \frac{\bar{H}}{n\sinh(au)} u_i = -h^k_i u_k v^2.
\end{equation}
Since $(h^i_j)$ is positive definite, this implies $Du = 0$. Hence $v=1$ and the proof of the Lemma follows by taking the limit $\delta\rightarrow 0$.
\end{proof}

Next, we want to establish uniform bounds on the outer radius and the inradius of $M_t$ for $t\in [0, T^*)$. Herefore we need some results on $h$-convex domains in $\bbh^{n+1}_{\frac{1}{a}}$.

\begin{Theorem}
\label{VolumeThm}
Let $\Omega$ be a $h$-convex domain in $\bbh^{n+1}_{\frac{1}{a}}$ and denote the center of an inball by $p$ and its radius by $\rho$. Furthermore let $\tau := \tanh(a\frac{\rho}{2})$.
	We have the inequality
	\begin{equation}
	\label{PinchingRadius}
		\textnormal{maxd}(p, \partial \Omega) - \rho \leq a\log\frac{(1+\sqrt{\tau})^2}{1+\tau} < a\log 2. 
	\end{equation}
	Therefore there exists a constant $c=c(a)>0$ such that
	\begin{equation}
	\label{PinchingRadius2}
		R \leq c(\rho + \rho^{\frac12}),
	\end{equation}
	where $R$ denotes the outer radius of $\Omega$.
\end{Theorem}
\begin{proof}
The proof of \eqref{PinchingRadius} can be found in \cite[Theorem 3.1]{BorMiqCompHad}. To prove \eqref{PinchingRadius2} we note $R\leq \maxd(p, \partial\Omega)$ and obtain from inequality \eqref{PinchingRadius}
\begin{equation}
	e^R \leq e^\rho \cdot \left( 1+ 2\sqrt \tau\right)^a.
\end{equation}
If $\frac{a\rho}{2} \geq \frac14$, then $(1+ 2\sqrt \tau)^a \leq 3^a \leq (6 a\rho)^a \leq e^{c\rho}$ with some constant $c=c(a)$. This implies $R\leq cÊ\rho$.
 
On the other hand, if $\frac{a\rho}{2} < \frac14$, then by using $e^x-e^{-x} = e^{-x} (e^{2x} -1) \leq \frac{2x}{1-2x} \leq 4x$ for $x<\frac14$, we obtain from the Bernoulli inequality
\begin{equation}
	(1+ 2\sqrt{\tau})^a \leq (1 + 4\sqrt{a\rho})^aÊ\leq e^{4a\sqrt{a\rho}}. 
\end{equation}
This implies \eqref{PinchingRadius2} with $c= 1+4a^{\frac32}$.
\end{proof}

We will also need the following monotonicity of mixed volumes:
\begin{Lemma}
\label{MixedVolumeMonotonicity}
Let $A\subset B\subset \bbh^{n+1}_{\frac1a}$ be convex domains with smooth boundary. Then there holds for all $k\in \{1, \ldots, n+1\}$
\begin{equation}
	V_k(A) \leq V_k(B).
\end{equation}
\end{Lemma}
\begin{proof}
	If $k=n+1$, the inequality follows directly from the definition of $V_{n+1}$. Hence let $1\leq k\leq n$. Then the inequality follows from \cite[Corollary 3.4.7]{Solanes}. 
\end{proof}

As a consequence of Theorem \ref{VolumeThm} and the monotonicity of mixed volumes we obtain:
\begin{Korollar}
\label{BoundRho}
Denote by $\rho_t$ the inradius and by $R_t$ the outer radius of the hypersurfaces $M_t$. Then there exists a constant $0< c_1 = c_1 (a, V_{n+1-k}(M_0))$ such that for $t\in [0, T^*)$ we have
\begin{equation}
\label{BoundsRadii}
c_1^{-1} \leq \rho_t \leq R_t \leq c_1.
\end{equation}
This also implies for $t\in [0, T^*)$ and $p,q \in \Omega_t$ 
\begin{equation}
	\dist(p,q) \leq 2 c_1.
\end{equation}
\end{Korollar}
\begin{proof}
This follows from Theorem \ref{VolumeThm} by noting that for the flow \eqref{floweq} with $f=f_k$, $k\in \{0, \ldots, n\}$, the mixed volume $V_{n+1-k}$ is preserved. We argue as in \cite[Corollary 3.6]{McCoyMixedAreaGen} and show only the lower bound in \eqref{BoundsRadii}, the upper bound follows analogously. 

Let $r=r(V_{n+1-k}(M_0))$ be such that \mbox{$V_{n+1-k}(B_r) = V_{n+1-k}(M_0)$}, where $B_r$ denotes a geodesic ball of radius $r$. Due to Lemma \ref{MixedVolumeMonotonicity} we obtain $r\leq R_t$ for all $t\in [0, T^*)$. Assume $\rho_t\leq 1$, for otherwise the lower estimate on $\rho_t$ is trivial. Inequality \eqref{PinchingRadius2} implies
\begin{equation}
	\rho_t \geq \left(\frac{R_t}{2c}\right)^2 \geq \left(\frac{r}{2c}\right)^2 =: c_1^{-1}. 
\end{equation}
\end{proof}

Now we have everything we need to get a uniform bound for the curvature function. We argue similarly as in \cite[Section 7]{CabMiqHyp}.
\begin{Theorem}
There exists $c_2 = c_2(n, a^2, M_0)>0$ such that
\begin{equation}
\label{BoundF}
	\sup_{t\in [0, T^*)}\sup_{x\in M_t} F(x) \leq c_2.
\end{equation}
\end{Theorem}
\begin{proof}
	Let $t_0 \in [0, T^*)$, and let $M_{t_0}$ be represented as a graph in geodesic polar coordinates centered at the center of an inball of $M_{t_0}$, $M_{t_0} = $ graph $u_{|\bbs^{n}}$. Corollary \ref{BoundRho} implies $c_1 \geq u(t_0) \geq c_1^{-1}$. Hence from Lemma \ref{LowBoundGraph} we confer that for small $\delta > 0$ we can choose $\beta := \log \cosh\tfrac{a}{c_1}$ such that the graph representation is valid for $t \in  [t_0, \underbrace{\min\{t_0 + \tfrac{\beta}{2a^2c_0}, T^*-\delta\}}_{=:t_1}]$ and 
	\begin{equation}
		\frac{\beta}{2a} \leq u(t) \leq 2c_1. 
	\end{equation}	
	
Let $\gamma:= \tfrac{1}{2}\sinh \tfrac{\beta}{2}$ and define $\eta := \frac{1}{\chi^{-1} - \gamma}$, then for $t \in [t_0, t_1]$ we obtain 
\begin{equation}
\label{BoundEta}
	\frac{1}{\sinh(2c_1a) - \gamma}\leq\eta(t) \leq \frac{1}{\gamma}
\end{equation}
 in view of \eqref{BoundChiEq}.
 
 We also obtain the following evolution equation for $F\eta$:
\begin{equation}
\begin{split}
\label{EvFEta}
L(F\eta) &= - F^{ij}h_i^kh_{kj} F\eta^2\gamma -f\eta F^{ij}h^k_ih_{kj} + a^2 \eta (f-F) F^{ij}g_{ij} \\
&+ 2F^{ij}(F\eta)_i \eta_j\eta^{-1} + (2F+\alpha F^{ij}g_{ij}-f) a \cosh(au) F \eta^2.
\end{split}
\end{equation}

Let $t\in [t_0, t_1]$ and $x_0 \in M_t$ be given such that
\begin{equation}
	\sup_{t\in [t_0,t_1]}\sup_{x\in M_t} (F\eta)(x) = (F\eta)(x_0).
\end{equation}
We introduce Riemannian normal coordinates at $x_0$ such that the principal curvatures are monotonically ordered, $\kappa_1 \leq \kappa_2\leq \ldots\leq \kappa_n$. Then we use the maximum principle and infer
\begin{equation}
\begin{split}
0 &\leq \frac{d}{dt} (F\eta) - F^{ij}(F\eta)_{;ij}  = - F^{ij}h_i^kh_{kj} F\eta^2\gamma -f\eta F^{ij}h^k_ih_{kj} \\
&+ a^2 \eta (f-F) F^{ij}g_{ij} + 2F^{ij}(F\eta)_i \frac{\eta_j}{\eta} + (2F+\alpha F^{kl}g_{kl}-f)a \cosh(au) F \eta^2.
\end{split}
\end{equation}
Now \eqref{PresPinchEq} yields
\begin{equation}
	F^{ij}h^k_ih_{kj} \geq F \kappa_1 \geq \eps F^2.
\end{equation}
We infer the inequality
\begin{equation}
0\leq -\eps \gamma F^3\eta^2 + (2F +\alpha F^{kl}g_{kl}) F \eta^2 a \cosh(2ac_1),
\end{equation}
which implies
\begin{equation}
	(F\eta)(x_0) \leq \max\left\{\frac{3a\cosh(2ac_1)}{\eps\gamma}, nc_0\alpha, \sup_{M_{t_0}} (F\eta)\right\}.
\end{equation}

The estimate \eqref{BoundF} then follows from \eqref{BoundEta}, taking the limit $\delta \rightarrow 0$ and the fact that $t_0$ can be chosen arbitrarily in $[0, T^*)$.
\end{proof}

The boundedness of $F$ implies the boundedness of $|A|^2$ due to the curvature pinching, as we will show in the following
\begin{Korollar}
	There exists $c_3 = c_3(n, c_2, \alpha) > 0$, such that 
	\begin{equation}
		|A| \leq c_3.
	\end{equation}
\end{Korollar}
\begin{proof}
	We distinguish two cases:
	
	Firstly, assume $F$ is a convex curvature function. Using \eqref{FHIneq} and the convexity of the hypersurfaces we obtain
	\begin{equation}
		F = \tilde{F} \circ \eta \geq \frac{1}{n} H\circ \eta = \frac1n H - \alpha \geq \frac1n \kappa_n - \alpha.
	\end{equation}
	The boundedness of $|A|$ now follows from the boundedness of $F$.
	
	Now assume $F$ is a concave curvature function. Then we infer from the curvature pinching \eqref{PresPinchEq} and from \eqref{FGIneq}
	\begin{equation}
		\kappa_n - a \leq \frac{1}{\tilde{\eps}} (\kappa_1 - a) \leq \frac{1}{\tilde{\eps}}(F - (a-\alpha)).
	\end{equation}
	Again the boundedness of $|A|$ follows from the boundedness of $F$.
\end{proof}

\section{Long time existence of the flow}

It remains to show that we have a uniform lower bound on the curvature function to infer the long time existence of the flow.
The following Lemma together with Lemma \ref{PresPinch} justifies the assumption of strict $h$-convexity during the flow. Note that for small times strict $h$-convexity holds due to the smoothness of the flow for small times and the strict $h$-convexity of the initial hypersurface.

Firstly, we want to cite a corollary of the parabolic Harnack inequality, which will allow us to estimate $F - (a-\alpha)$ uniformly from below.

\begin{Proposition}
\label{HarnackIneq}
For $(x_0, t_0) \in \bbr^n\times \bbr$ and $r\in \bbr_+$ we let $Q((x_0, t_0), r) := B_{r}(x_0) \times [t_0-r^2, t_0] \subset \bbr^n\times \bbr$ and $Q(r) := Q((0,0), r)$. Let $u \in C^{\infty}(Q(4 R))$ be a nonnegative solution of an equation of the form
\begin{equation}
Lu = - \dot u + a^{ij}u_{ij} + b^i u_i = f.
\end{equation}
Here $f=f(x, t, u(x, t))$ and we assume that there exists $\alpha \in \bbr_+$ so that $f$ satisfies the inequality $-\alpha u(x, t) \leq f(x, t, u(x, t)) \leq \alpha u(x, t)$ for all $(x, t)\in Q(4R)$. We assume the coefficients are measureable and bounded by a constant $c_0 \in \bbr_+$ and there exist $0<\lambda\leq \Lambda <\infty$ such that $\lambda (\delta^{ij}) \leq (a^{ij}) \leq \Lambda (\delta^{ij})$. 
Then there exists $c=c(n, \lambda, \Lambda, R, \Vert b\Vert_{L^{\infty}}, c_0, \alpha) > 0$ such that there holds
\begin{equation}
	\underset{Q\left((0, -4R^2), \frac R2\right)}{\sup} u \leq c\cdot \underset{Q(R)}{\inf} u.
\end{equation}
\end{Proposition} 
\begin{proof}
We apply \cite[Theorem 7.36]{Lieberman} to the function $u$ and \cite[Theorem 7.37]{Lieberman} to the function $\eta := e^{\alpha t} u$.
\end{proof}

We infer the following
\begin{Lemma}
\label{FLowerBound}
	There exists a constant $0< c_5 = c_5(c_1, c_2, c_3, \tilde \eps)$ such that for all $t\in [0, T^*)$
	\begin{equation}
		F- (a-\alpha) \geq c_5.
	\end{equation}
\end{Lemma}
\begin{proof}
	We will use an idea from \cite[Section 7]{AndrewsContractionEuc} in the euclidean setting.  Let $0<T\leq T^*$ be the maximal time such that the hypersurfaces $M_t$ remain strictly $h$-convex up to time $T$. For $t\in [0, T)$ let $x_t\in M_t$ be a point in contact with an enclosing sphere of radius $\frac{1}{2c_1}<\rho < 2 c_1$. This implies 
	\begin{equation}
	\label{BoundFSupBelow}
		\underset{x\in M_t}{\sup} F(x) \geq F(x_t) \geq a \coth (a \rho) \geq a \coth(\frac{a}{2c_1}).
	\end{equation}
	Next we note that $\eta:= F- (a-\alpha)$ satisfies the evolution equation
	\begin{equation}
		L(F- (a-\alpha)) = (F-f) \left(F^{ij}h_i^kh_{kj} - a^2 F^{ij}g_{ij}\right).
	\end{equation}
	In the case of a convex curvature function we have in view of \eqref{FGIneq}
	\begin{equation}
		0\leq \sum_{i=1}^n f_i (\kappa_i^2 -a^2) \leq c \sum_{i=1}^n f_i (\kappa_i -\alpha - (a-\alpha))Ê\leq c (F- (a-\alpha)).
	\end{equation}
	In the case of a concave curvature function we have in view of the pinching estimate and \eqref{FGIneq}
	\begin{equation}
	\begin{split}
		0\leq &\sum_{i=1}^n f_i (\kappa_i^2 -a^2) \leq c (H - na) \leq \frac{c}{\tilde{\eps}}
		\left(\sum_{i=1}^n f_i (\kappa_1 - \alpha) - (a-\alpha)\right) \\&\leq \frac{c}{\tilde{\eps}} (F - (a-\alpha)).
	\end{split}
	\end{equation}
	Since $L$ is uniformly parabolic in view of Corollary \ref{UniformlyParabolic}, we can apply Proposition \ref{HarnackIneq} together with \eqref{BoundFSupBelow} to obtain the desired lower bound for \mbox{$F- (a-\alpha)$} up to time $T$. This also implies $T = T^*$ in view of the pinching estimate.
\end{proof}

Hence we know that as long as the hypersurfaces are strictly $h$-convex, Lemma \ref{PresPinch} and Lemma \ref{FLowerBound} are valid. This implies that the hypersurfaces remain uniformly strictly $h$-convex up to $t=T^*$.

Finally, we want to establish the higher order estimates to obtain the long time existence of the flow. At this point we have obtained uniform $C^2$-estimates for $u$ in $[0, T^*)$, which are independent of $T^*$, and we have shown the uniform ellipticity of the operator $F^{ij}$ (see Corollary \ref{UniformlyParabolic}). To obtain higher order estimates, uniformly in time, we can follow the same procedure as in \cite[Section 8]{McCoyMixedAreaGen} and \cite[Section 6]{RivSin}, see also \cite[Section 8]{MakVolLor} for a more detailed account of this procedure. This shows the long time existence of the flow by standard arguments, see for example \cite[Section 2.6]{GerhCP}.

\section{Convergence to a geodesic sphere}

To prove the convergence of the flow to a geodesic sphere, firstly we will show that the pinching of the principal curvatures is improving at an exponential rate. Then we will use the argument from \cite[Theorem 3.5]{SchulzeConvexity} to obtain the exponential convergence of the flow to a geodesic sphere.

\begin{Proposition}
\label{ImprovedPinching}
	There exists $\lambda > 0$ and $t_0>0$ such that we have for all $t\in [t_0, \infty)$ at points $x\in M_t$
	\begin{equation} 
	\label{ImprovedPinchingEq}
		\kappa_1 - a\geq \frac{1}{n}(1-e^{-\lambda t}) (H- na),
	\end{equation}
	where we denote by $\kappa_1$ the smallest principal curvature of $M_t$ at $x$.
\end{Proposition}
\begin{proof}
	Firstly, we assume $F$ to be a convex curvature function. We define $S_{ij} = h_{ij} - (a + (1-e^{-\lambda t})\, (F - (a-\alpha))) g_{ij}$, where $\lambda > 0$ is a small number yet to be chosen. We use Theorem \ref{AndPinch}, however we start at time $t_0 :=\frac{\zeta}{\lambda}$ instead of $t_0 = 0$, where $\zeta>0$ is a constant chosen such that $\eps \geq 1- e^{-\zeta}$ and $\eps$ is chosen as in Lemma \ref{PresPinch}. As in the proof of Lemma \ref{PresPinch} (we also use analogous notation as in that Lemma), we obtain with $\bar\eps := 1-e^{-\lambda t}$
	\begin{equation}
		N_{ij}v^iv^j \geq 2\bar\eps a\, e^{-\lambda t} (F-(a-\alpha))^2 - \lambda e^{-\lambda t} (F - (a-\alpha))\geq 0,
	\end{equation}
	if we choose $\lambda > 0$ small enough depending on $c_5$. Hence we obtain the desired inequality in view of inequality \eqref{FHIneq}.

	Now we assume $F$ to be a concave and inverse concave curvature function. Let $\lambda > 0$ be a small number depending only on $c_5$. Let $S_{ij} = h_{ij} + (-a - \eps H + \eps a\, n)g_{ij}$ with $\bar\eps := \frac1n(1- e^{-\lambda t})$. Again we obtain
	\begin{equation}
		N_{ij}v^i v^j \geq \bar\eps^2\, e^{-\lambda t} a (H- an)^2 \sum_{i=1}^n f_i - \tfrac\lambda n e^{-\lambda t} (H - an).
	\end{equation}
	Since $H- an \geq n (F- (a-\alpha))$ in view of \eqref{FHIneq} and $F^{ij}g_{ij} \geq 1$ in view of \eqref{FGIneq} we obtain the inequality \eqref{ImprovedPinchingEq} from Theorem \ref{AndPinch} (again starting at time $t_0 := \frac{\zeta}{\lambda}$ with $\lambda$ small enough and $\zeta$ as above).
\end{proof}

From the preceding Proposition we can conclude with the same arguments as in \cite[Theorem 3.5]{SchulzeConvexity} that the flow converges exponentially in $C^\infty$ to a geodesic sphere:
\begin{Korollar}
	There exists $t_0 >0$ and positive constants $C$, $r_0$, $\delta_i$, $C_i$ for $iÊ\in \bbn_+$ such that for all $t\in [t_0, \infty)$ the hypersurfaces $M_t$ can be written as graphs over a geodesic sphere, $M_t = $ graph$_{|\bbs^n}u$, and there holds
	\begin{align}
	\label{ImprovedPinchingEq2}
		|A|^2 - \frac{H^2}{n} &\leq C_0 e^{-\delta_0 t},\\
	\label{DerivativesExpEq}
		\Vert \nabla^{(i)} A\Vert &\leq C_i e^{-\delta_i t},\\
	\label{FExpEq}
		|F - f| &\leq C e^{-\delta_1 t}\\
	\label{GraphExpEq}
		|u - r_0| &\leq C e^{-\delta_1 t}.
	\end{align} 
	Hence the flow converges exponentially in $C^\infty$ to a geodesic sphere of radius $r_0$, which is determined by $V_{n+1-k}(M_0)$.
\end{Korollar}
\begin{proof}
	The estimate \eqref{ImprovedPinchingEq2} follows directly from \eqref{ImprovedPinchingEq}. By interpolation we obtain the estimates \eqref{DerivativesExpEq} (see the proof of \cite[Theorem 3.5]{SchulzeConvexity}) for $i\in \bbn_+$. The estimate \eqref{FExpEq} follows from \eqref{DerivativesExpEq} for $i=1$ and the boundedness of $\rho_t$ (which implies the boundedness of $\diam(M_t)$). Now since $|F-f|$ is integrable over time and $\rho_t \geq c_1^{-1}$ we know there exists $t_0 \in [0, \infty)$, such that $M_t$ can be represented as graph $u$ for $t\in [t_0, \infty)$. The last estimate \eqref{GraphExpEq} then follows from \eqref{FExpEq} and \eqref{dotU}.
\end{proof}

\section{Volume inequalities in hyperbolic space}

In this section we note, that we can use an idea from \cite[Section 10]{McCoyMixedAreaGen} to prove volume inequalities in hyperbolic space for strictly $h$-convex hypersurfaces. We only give the easiest example of how to use the volume preserving curvature flows to obtain such inequalities.

\begin{Korollar}
Let $M_0$ be a strictly $h$-convex hypersurface in hyperbolic space. Let $R_0 > 0$ be such that a geodesic sphere of radius $R_0$ satisfies $V_{n+1}(M_0) = V_{n+1}(B_{R_0})$. Then there holds
\begin{equation}
	\frac{V_{n+1}(M_0)}{|M_0|} \leq \frac{V_{n+1}(B_{R_0})}{|B_{R_0}|}.
\end{equation} 
\end{Korollar}
\begin{proof}
	We use the volume preserving curvature flow with $F=H$ and obtain that
	\begin{equation}
		\tfrac{d}{dt} |M_t| \leq 0,
	\end{equation}
	in view of the H\"older inequality. Since $M_t$ converges to a geodesic sphere of radius $R_0$, we obtain $|M_0| \geq |B_{R_0}|$, showing the claimed inequality.	
\end{proof}

Unfortunately, we were not able to prove all Minkowski inequalities (only some further special cases), due to the fact, that the volume preserving term has a different structure than in the euclidean case.

\bibliographystyle{plain}
\bibliography{Bibliography}

\end{document}